\documentclass[12pt]{amsart}

\newtheorem{theorem}{Theorem}
\newtheorem{thm}{Theorem}[section]
\newtheorem{cor}[thm]{Corollary}
\newtheorem{lem}[thm]{Lemma}
\newtheorem{conj}[thm]{Conjecture}

\def\ff#1{{\mathbb F}_{#1}}
\def\ffs#1{{\mathbb F}_{#1}^\ast}
\def\ffx#1{\ff{#1}[X]}

\newcommand{\Tr}{{\text {Tr}}}

\newcommand{\itr}{\mathbb Z}
\newcommand{\nat}{\mathbb N}

\newcommand{\D}{{\rm Degree}}
\newcommand{\M}{{\rm Max}}

\begin{document}
\title[On the number of distinct values of a class of functions]{On the number of distinct values of a class of functions with finite domain}

\author[R.S. Coulter]{Robert S. Coulter}

\address{Ewing Hall\\
Department of Mathematical Sciences\\University of Delaware\\
Newark, DE 19716, USA}

\author[S. Senger]{Steven Senger}

\address{Ewing Hall\\
Department of Mathematical Sciences\\
University of Delaware\\
Newark, DE 19716, USA}



\begin{abstract}
By relating the number of images of a function with finite domain to a certain
parameter, we obtain both an upper and lower bound for the image set.
Even though the arguments are elementary, the bounds are, in some sense, best
possible.
These bounds are then applied in several contexts.
In particular, we obtain the first non-trivial upper bound for the image set
of a planar function over a finite field.
\end{abstract}

\maketitle

\section{Introduction}

Let $A$ and $B$ be sets, with $A$ finite of order $n$, and let
$f:A\rightarrow B$. We define the following notation, which will be used
throughout this article.
\begin{itemize}
\item The number of distinct images of $f$ is denoted by $V(f)$. That is,
$V(f)=|f(A)|$.

\item For $r\in\nat$, $M_r(f)$ is the number of $y\in B$ for which $f(x)=y$ has
$r$ solutions.

\item Since $A$ is finite, clearly $M_r(f)=0$ for all sufficiently large $r$.
We therefore define $m$ to be the largest integer for which $M_m>0$.

\item For each integer $r\ge 2$, $N_r(f)$ is the number of $r$-tuples
$(x_1,\ldots,x_r)$ with $x_i=x_j$ if and only if $i=j$ which satisfy
$f(x_1)=f(x_2)=\cdots=f(x_r)$.
\end{itemize}
Several identities follow immediately from these definitions.
\begin{enumerate}
\renewcommand{\labelenumi}{Id\#\arabic{enumi}}
\item $V(f) = \sum_{r=1}^m M_r(f)$.
\item $n = \sum_{r=1}^m r M_r(f)$.
\item $N_s(f) = \sum_{r=s}^m P(r,s) M_r(f)$.
\end{enumerate}
(Here $P(r,s)$ denotes the number of $s$-permutations from $r$ distinct
objects. Recall $P(r,s)=0$ when $r<s$.)

In this paper we are interested in the relationship between $V(f)$ and
$N_s(f)$ for a fixed $s$. Intuitively, knowledge of $N_s(f)$ should imply some
knowledge on $V(f)$, and knowledge of $N_s(f)$ should yield more knowledge
concerning $V(f)$ than $N_{s'}(f)$ would for $s'>s$.
Our main result is to obtain bounds for $V(f)$ in terms of $N_s(f)$ which
confirm this intuition.
Moreover, when $s=2$, our lower bound is tight for any value of $N_2(f)$, while
our upper bound is tight in infinitely many cases.
Our main theorem can be given in the following form.
\begin{theorem} \label{mainthm}
Let $f:A\rightarrow B$ with $|A|=n$. Then
\begin{equation*}
\frac{1}{s-1}\left(n - \frac{N_s(f)}{s!}\right)
\le V(f) \le n - N_s(f)^{1/s} + O(N_s(f)^{1/(s+1)}).
\end{equation*}
\end{theorem}
We pay particular attention to the case $s=2$ because it is more likely that
one has information on pairs of elements with the same image than, say,
$3$-tuples or $4$-tuples.
In addition, the upper bound can be made explicit in this case.
\begin{theorem} \label{seq2thm}
Let $f:A\rightarrow B$ with $|A|=n$ and set $N_2(f)=t$.
Then $M_1(f)\ge  \M(0,n-t)$ and
\begin{equation*}
n - \frac{t}{2} \le M_1(f)+M_2(f)\le V(f) \le n - \frac{2t}{1+\sqrt{4t+1}}
\end{equation*}
\end{theorem}
Interestingly, the upper bound in Theorem \ref{seq2thm} is related to
triangular numbers, and a slight improvement of this bound, in some cases,
could be obtained by resolving a problem on them.

Theorems \ref{mainthm} and \ref{seq2thm} can be applied in a variety of
settings.
We choose to limit ourselves to just one main application -- to polynomials
over finite fields.

Let $q$ be a positive power of some prime $p$.
We use the standard notation of $\ff{q}$ for the finite field of $q$ elements,
$\ffs{q}$ for the non-zero elements of $\ff{q}$, and $\ffx{q}$ for the ring of
polynomials over $\ff{q}$ in $X$.
We prove that for a polynomial $f\in\ffx{q}$, the expected value of
$N_2(f)$ is $q-1$.
Consequently, we obtain the following corollary to Theorem \ref{seq2thm}.
\begin{theorem} \label{polyversion}
Suppose $f\in\ffx{q}$ is a polynomial for which $N_2(f)=q-1$, the expected
value.
Then
\begin{equation*}
\frac{q+1}{2} \le V(f) \le q - \frac{2(q-1)}{1+\sqrt{4q-3}}.
\end{equation*}
\end{theorem}
Several classes of polynomials which obtain the expected value for $N_2(f)$ are
then described; these include the class of planar polynomials (for further
definitions, see Section \ref{polysection}).
Planar polynomials are closely related to affine planes
\cite{coulter97a, dembowski68}, semifields \cite{coulter08}, and difference
sets \cite{ding06, qiu07}.
Consequently, they have received a significant amount of attention.
However, the bound given by Theorem \ref{polyversion} constitutes the first
non-trivial upper bound obtained on the size of the image set of a planar
function.
We suspect that, for planar functions, our upper bound can still be improved
as we do not utilise the full set of restrictions implied by the planar
property.
The lower bound is, for planar functions, tight, and has been derived
previously by several authors, see \cite{coulter11, kyureghyan08, qiu07}.
Our result, in this sense, constitutes a generalisation of the respective
results given in each of those three papers.

The paper is set out as follows.
In the next section we prove Theorems \ref{mainthm} and \ref{seq2thm}.
We also discuss briefly the connection between Theorem \ref{seq2thm} and
triangular numbers.
In Section 3 we apply our results to polynomials over finite fields.
The paper ends with some observations in arithmetic combinatorics and
coding theory.

\section{Bounding $V(f)$ when $N_s(f)$ is known}

For convenience, we set $N_s(f) = t$. By the definitions above,
\begin{equation}
\sum_{r=1}^{s-1} r M_r = n - t + \sum_{r=s}^m \left(P(r,s)-r\right) M_r.
\label{aneq}
\end{equation}
(We note that, since the sum on the right is at least $m(m-2)$, we must have
$\sum_{r=1}^{s-1} r M_r \ge \M(0, n - t + m(m-2))$.)
We may manipulate (\ref{aneq}) as follows:
\begin{align*}
\sum_{r=1}^{s-1} r M_r
&= n - t + \sum_{r=s}^m \left(P(r,s)-r\right) M_r\\
&= n - t + (s! -s) M_s + \sum_{r=s+1}^m \left(P(r,s)-r\right) M_r\\
&\ge n - t + (s!-s) M_s + (s!-1) \sum_{r=s+1}^m r M_r\\
&= n - t + (s!-s) M_s\\
&\quad+ (s!-1) \sum_{r=1}^m r M_r
- (s!-1)\sum_{r=1}^{s-1} rM_r - (s!-1) sM_s\\
&= s!\, n - t +  s!\, (1-s) M_s - (s! -1) \sum_{r=1}^{s-1} r M_r.
\end{align*}
Rearranging, we find
\begin{align*}
s! \, n - t &\le s! \sum_{r=1}^{s-1} rM_r + s!\,(s-1) M_s\\
&\le s!\, (s-1) \sum_{r=1}^{s-1} M_r + s!\,(s-1)M_s\\
&= s!\, (s-1) \sum_{r=1}^s M_r\\
&\le s!\, (s-1) \, V(f),
\end{align*}
which establishes the lower bound in Theorem \ref{mainthm}.
(We mention, in passing, that this proof is a generalisation of the lower bound
obtained by Matthews and the first author \cite{coulter11}; it was that note
that formed the motivation for this article.)

We now move to determine the upper bound.
First, we note that $M_m>0$, and so $P(m,s)\leq t$, which yields
\begin{equation}\label{tUpper}
m\leq t^\frac{1}{s}+O(t^\frac{1}{s+1}).
\end{equation}
Now, we apply the definitions above to obtain
\begin{align*}
t &= N_s(f)\\
&= \sum_{r=s}^m P(r,s)M_r
= \sum_{r=1}^m P(r,s) M_r\\
&\leq m\sum_{r=1}^m P(r-1,s-1) M_r\\
&\leq m\cdot P(m-2,s-2) \sum_{r=1}^m (r-1)M_r,\\
\end{align*}
from which we deduce
\begin{equation}\label{mUpper}
\sum_{r=1}^m (r-1) M_r \geq \frac{t}{m\cdot P(m-2,s-2)}.
\end{equation}
Combining \eqref{tUpper} and \eqref{mUpper}, we get
\begin{equation}\label{sumBound}
\sum_{r=1}^m (r-1) M_r \geq t^\frac{1}{s} - O(t^\frac{1}{s+1}).
\end{equation}
We can now estimate $V(f)$ using this sum:
\begin{align*}
V(f) &= n - n + V(f)\\
&= n - \sum_{r=1}^m r M_r - \sum_{r=1}^m M_r\\
&= n - \sum_{r=1}^m (r-1) M_r\\
\end{align*}
Applying \eqref{sumBound} yields
\begin{equation}\label{mainUpper}
V(f) \leq n - t^\frac{1}{s} + O(t^\frac{1}{s+1}),
\end{equation}
as claimed.

The proof of Theorem \ref{seq2thm} is no more difficult; in fact, the lower
bound is precisely that from before, while the upper bound follows
from a careful re-working of the proof of the upper bound.
We omit the details.

It is easy to see that, provided $N_2(f) <  2n$, this lower bound is tight, as
one can easily construct functions that meet this bound.
Set $N_2(f)=t$.
Randomly choose $t$ distinct elements $x_1,x_2,\ldots,x_t\in A$ and
$t/2$ distinct elements $y_1,y_2,\ldots,y_{t/2}\in B$.
For $1\le i\le t/2$, assign $f(x_{2i-1})=f(x_{2i})=y_i$.
At this point, $N_2(f)=t$, so that $f$ must be 1-1 on
$A\setminus\{x_1,\ldots,x_t\}$.
It follows that $V(f)=\frac{t}{2} + n-t = n - \frac{t}{2}$, which is the lower
bound.

It is clear from symmetry that $N_2(f)=t$ is necessarily even. Set $t=2k$.
Then the bounds read
\begin{equation*}
n-k\le V(f) \le n - \frac{4k}{1+\sqrt{8k+1}}.
\end{equation*}
It is natural to ask when is $\sqrt{8k+1}\in\itr$? Interestingly, $8k+1$ is
a square precisely when $k$ is a triangular number. In such cases, we have
$k=u(u-1)/2$ for some integer $u$, $8k+1=\delta^2$ where $\delta=2u-1$, and
the upper bound simplifies neatly to
\begin{equation*}
V(f) \le n - \frac{\delta-1}{2} = n + 1 - u.
\end{equation*}
In all cases where $k$ is a triangular number, there exist functions which
attain this bound.
To construct such a function, choose $u$ elements $x_1,x_2,\ldots,x_u\in A$ and
set $f(x_1)=f(x_2)=\cdots=f(x_u)$.
Now set $f$ to behave 1-1 on the remaining elements of $A$.
It can be seen that $N_2(f)=2k$ and that the upper bound is attained.

In all cases where $k$ is not a triangular number, our upper bound is not
exact.
To make our upper bound tight, one needs to solve the following problem:
\begin{quote}
Let $T_r=\binom{r}{2}$ for any $r\in\nat$, and fix $k\in\nat$.
By a {\em triangular sum of length $l$ for $k$} we mean any instance of the
equation
\begin{equation*}
k = \sum_{i=1}^l T_{r_i},
\end{equation*}
where $r_1\ge r_2\ge\cdots\ge r_l$.
The {\em weight} of a given triangular sum is given by $-l+(\sum_{i=1}^l r_i)$.
Given $k$, we define $B_k$ to be the smallest weight among all triangular sums
for $k$.
Find a formula for $B_k$.
\end{quote}
Clearly, when $k=T_u$, $B_k=u-1$, but we do not know of a general formula for
$B_k$.
While Gauss famously proved that there exists a triangular sum for
any $k$ with length at most 3, it may not necessarily be the case that one such
instance will provide the value for $B_k$.
The connection to our bound should be clear: If $N_2(f)=2k$, then
$V(f)\le n - B_k$, with equality always possible.

\section{Polynomials over finite fields and $N_2(f)$} \label{polysection}

We now look to apply these bounds on $V(f)$ to polynomials over finite
fields.
It is, of course, well known that every function over $\ff{q}$ can be
represented uniquely, via Lagrange interpolation, by a polynomial of degree
less than $q$.
By the {\em reduced form} of a polynomial $f\in\ffx{q}$ we shall mean the
polynomial $g(X)$ given by $g(X) = f(X) \bmod (X^q-X)$.
A polynomial $f\in\ffx{q}$ is a {\em permutation polynomial} over $\ff{q}$ if
$V(f)=q$.

Research concerning the value of $V(f)$ for polynomials over finite fields is
extensive; we restrict ourselves to discussing a few outstanding general
results. It is clear that, for lower bounds, there are obvious limits to the results you
can expect to obtain -- obviously $V(f)\ge 1$ with equality possible, while
for polynomials of given degree $d$, $V(f)\ge 1+\frac{q-1}{d}$ is clear.
That said, we have the following deep result by Cohen \cite{cohen73}
concerning the average lower bound of $V(f)$.
\begin{thm}[Cohen \cite{cohen73}] \label{cohensthm}
Let $f\in\ffx{q}$ be of the form
\begin{equation*}
f(X) = X^d + \sum_{i=1}^{d-1} a_i X^i.
\end{equation*}
Let $t$ be any integer such that $0\le t\le d-2$ and let
$a_{d-1},\ldots,a_{d-t}$ be fixed.
Define $v(d,t)=\sum V(f)/q^{d-t-1}$, where the sum is over all
$a_1,\ldots,a_{d-t-1}$.
Set $m=\lfloor (d-t)/2\rfloor$.
Then $v(d,t) > c(q,m) q$, where
\begin{equation*}
c(q,m) = 1 - \left(\sum_{r=0}^m \binom{q}{r} (q-1)^{-r}\right)^{-1}.
\end{equation*}
\end{thm}
Setting $t=d-2$ in Cohen's result, we find that, in particular, on average,
$V(f)>\frac{q^2}{2q-1}>\frac{q}{2}$.

A specific lower bound was obtained by Wan, Shiue, and Chen \cite{wan93b} under
an additional condition on the polynomial.
For $f\in\ffx{q}$, define $u_p(f)$ to be the smallest positive integer $k$
such that $\sum_{x\in\ff{q}} f(x)^k\ne 0$. If no such $k$ exists, define
$u_p(f)=\infty$.
\begin{thm}[Wan, Shiue, Chen \cite{wan93b}]
If $u_p(f)<\infty$, then $V(f)\ge u_p(f) + 1$.
\end{thm}
The authors note that $u_p(f)\ge \lfloor\frac{q-1}{\D(f)}\rfloor$, so that
under the conditions, their bound is at least as good as the obvious bound
noted above.

In terms of an upper bound, there is the following general bound by Wan
\cite{wan93a}, given in terms of the degree of the polynomial.
\begin{thm}[Wan \cite{wan93a}] \label{wansthm}
Let $f\in\ffx{q}$.
If $f$ is not a permutation polynomial over $\ff{q}$, then
\begin{equation*}
V(f)\le q - \left\lfloor\frac{q-1}{\D(f)}\right\rfloor.
\end{equation*}
\end{thm}
A better bound was obtained in \cite{wan93b} using
$p$-adic techniques. To avoid unnecessary technical details, we simply refer
the interested reader to \cite{wan93b}, Theorem 3.1.

Integral to applying our bounds is having knowledge of $N_s(f)$ for
some $s$.
For simplicity, we only discuss the case $s=2$ here. We do not feel this is
particularly limiting as, of the values of $N_s(f)$, knowledge of $N_2(f)$
seems most likely.
We approach this issue by first establishing the expected value of
$N_2(f)$ for any polynomial $f\in\ffx{q}$ and applying our bounds to
polynomials with this expected value.
We then consider classes of polynomials which meet this expected value.

Denote the standard trace mapping from $\ff{q}{}$ to $\ff{p}{}$ by $\Tr$.
Let $\omega$ be a primitive $p$th root of unity.
Recall that the canonical additive character, $\chi_1$, of $\ff{q}{}$ is
defined by $\chi_1(x)=\omega^{\Tr(x)}$ for any $x\in\ff{q}{}$, and that all
additive characters of $\ff{q}{}$ are given by $\chi_h(x)=\chi_1(hx)$ for any
$h\in\ff{q}{}$.
The following result is a straight generalisation of a result of Carlitz
\cite{carlitz55}.
\begin{lem} \label{averagelemma}
Given a random polynomial $f\in \ffx{q}$, the expected value of
$N_2(f)$ is $q-1$.
Equivalently, for any $f\in\ffx{q}$,
\begin{equation}\label{Navg}
\sum_{a\in \ff{q}} N_2(f(X)+aX) = q(q-1).
\end{equation}
\end{lem}
\begin{proof}
Fix a polynomial $f\in \ffx{q}$.
By the definitions above,
\begin{align*}
q(N_2(f)+ q)
&= q (|\lbrace (x,y): f(x)=f(y), x,y\in \ff{q}, x\neq y \rbrace| \\
&\quad + |\lbrace (x: f(x)=f(x), x\in \ff{q} \rbrace|)\\
&=\sum_{h\in\ff{q}} \sum_{x,y\in \ff{q}}\chi_h(f(x)-f(y)).
\end{align*}
To generate our average value for $N_2(f)$, we consider the average over the
set $\{ f(X)+aX \,:\, a\in\ff{q}\}$.
We have
\begin{align*}
\sum_{a\in\ff{q}} q&(N_2(f(X)+aX)+q)\\
&= \sum_{a\in\ff{q}}\sum_{h\in\ff{q}}
\sum_{x,y\in \ff{q}}\chi_h(f(x)-f(y)+a(x-y))\\
&=q^3+ \sum_{h\in\ffs{q}} \sum_{x,y\in \ff{q}} \chi_h(f(x)-f(y))
\sum_{a\in\ff{q}} \chi_h(a(x-y))\\
&=q^3+ \sum_{h\in\ffs{q}} \sum_{x\in\ff{q}} q\\
&= q^3+q^2(q-1),
\end{align*}
where, in the second to last line, we have exploited the fact
$\sum_{a\in\ff{q}} \chi(a(x-y))=0$ unless $x=y$.
Comparing the left and right hand sides yields
\begin{equation} 
\label{a_1Avg} \sum_{a\in\ff{q}} N_2(f(X)+aX) = q(q-1).
\end{equation}
The claimed expected value of $N_2(f)$ now follows at once, for we can, of
course, partition the set of polynomials into equivalence classes, with
two polynomials being equivalent if they differ only by a linear term $aX$:
the average value of $N_2(f)$ for the polynomials in any equivalence class is
$q-1$ by (\ref{a_1Avg}).
\end{proof}
Theorem \ref{polyversion} now follows at once from Theorem \ref{seq2thm} and
Lemma \ref{averagelemma}.

Now suppose $f\in\ffx{q}$ is a polynomial for which $N_2(f)=q-1$, the expected
value.
For our lower bound, we find $V(f)\ge \frac{q+1}{2}$, which is more or
less the same as that obtained by Cohen's result.
In the other direction, applying our upper bound to $f$, we find
\begin{equation*}
V(f) \le q - \frac{2(q-1)}{1+\sqrt{4q-3}}.
\end{equation*}
However, this cannot be compared directly to the result of Wan, for we
do not know if $N_2(f)=q-1$ has any direct implication on $\D(f)$.

Given Lemma \ref{averagelemma}, one obvious question arises: Is it possible to
describe classes of polynomials for which the expected value for $N_2(f)$ is
obtained? Are there natural conditions on $f$ which force $N_2(f)=q-1$?
We now discuss, for $q$ odd, several such conditions (the case $q$ even
is clearly impossible for $N_2(f)$ is necessarily even).

For any $a\in\ffs{q}$, we define the {\em difference polynomial},
$\Delta_{f,a}(X)=\Delta_a(X)$, to be the polynomial given by
$\Delta_a(X)=f(X+a)-f(X)$.
A polynomial $f\in\ffx{q}$ is {\em planar} over $\ff{q}$ if, for every
$a\in\ffs{q}$, the polynomial $\Delta_a(X)$ is a permutation polynomial over
$\ff{q}$.
An equivalent definition for planarity is that
$|S_h(f(X)+aX)|=|\sum_{x\in\ff{q}} \chi_h(f(x)+ax)|=\sqrt{q}$ for all
$a,h\in\ff{q}$, $h\ne 0$.

Consider the following conditions on a polynomial $f\in\ffx{q}$:
\begin{enumerate}
\renewcommand{\labelenumi}{$C_\arabic{enumi}$.}
\item $f$ is planar over $\ff{q}$.
\item For $h\in\ffs{q}$, $|S_h(f)|=|\sum_{x\in\ff{q}} \chi_h(f(x))|=\sqrt{q}$.
\item For all $a\in\ffs{q}$, the polynomial $\Delta_{f,a}(X)$ has a unique
root.
\item $N_2(f)=q-1$.
\end{enumerate}
Clearly, $C_1 \rightarrow C_2$ and $C_1 \rightarrow C_3\rightarrow C_4$.
It is shown in the proof of \cite{coulter11}, Theorem 1, that
$C_2\rightarrow C_4$, while a counting argument, also given in \cite{coulter11},
shows $C_1 \not\equiv C_2$.

The relationship between $C_2$ and $C_3$ is less clear.
Computations show that they are almost certainly inequivalent for sufficiently
large $q$.
Over $\ff{3}$, they are equivalent; over $\ff{5}$, they are not, though
$(C_2\land C_3)\rightarrow C_1$. For $q\in\{7,9\}$, they are inequivalent, and
\begin{itemize}
\item there exist polynomials which satisfy both $C_2$ and $C_3$ but not $C_1$;
for example, $f(X)=X^4+2X^2\in\ffx{7}$; and
\item there exist polynomials which satisfy one or other but not both
conditions; for example, with $g$ a primitive element of $\ff{9}$, $X^7+gX^2$
satisfies $C_2$ but not $C_3$, while $X^8+gX^2$ satisfies $C_3$ but not $C_2$.
\end{itemize}
This also shows $C_2\not\equiv C_4$ and $C_3\not\equiv C_4$.
We suspect that the following statement is true, though we have no direct
idea of how to establish it.
\begin{conj}
For any finite field of any characteristic,
the number of polynomials satisfying $C_3$ is greater than or equal to the
number of polynomials satisfying $C_2$.
\end{conj}

\section{Two further settings where the bounds apply}

We end by describing two settings where our results can be applied, and where
we suspect some refinements of our methods might lead to stronger results than
those we give here.

\subsection{Arithmetic combinatorics}

Here, we present a setting where $N_2$ arises rather naturally.
Let $G$ be a (not necessarily abelian) group.
For subsets $A,B \subset G$, define the product set of $A$ and $B$ to be
$$A\cdot B = \lbrace ab : a\in A, b \in B \rbrace.$$
Much interest revolves around the relative sizes of $A, B$, and $A\cdot B$.
Some examples are the Cauchy-Davenport Theorem, the Pl\"{u}nnecke-Rusza
inequalities, and Freiman's Theorem; see the books by Nathanson
\cite{bnathanson96ii} or Tao and Vu \cite{btaovu}.
One useful tool for these questions is the concept of energy.
Various types of energy bounds have been the key ingredient in many recent
results, such as the current best known sums and products bound due to
Solymosi \cite{solymosi08}.

Given $G, A, $ and $B$ as above, we define the multiplicative energy, $E(A,B)$,
to be
$$E(A,B)= |\lbrace(a,a',b,b')\in A\times A\times B\times B:ab=a'b' \rbrace|.$$
If we consider $f:A\times B \rightarrow G, f:(a,b) \mapsto ab,$ we get a very
close relationship between $N_2(f)$ and $E(A,B)$, namely
$$N_2(f) = E(A,B) - |A|\cdot |B|,$$
which we obtain by removing the ``diagonal" elements of the form $(a,a,b,b)$
from the energy count.
With this in mind, the following is a direct application of Theorem
\ref{seq2thm}.
\begin{cor}
Let $G$ be a group, $A,B \subset G$ and set $n = |A| \cdot |B|$.
Then we have
\begin{equation}\label{energyBounds}
\frac{3n-E(A,B)}{2} \leq |A \cdot B|\leq
n - \frac{2(E(A,B)-n)}{1+\sqrt{4(E(A,B)-n)+1}}
\end{equation}
\end{cor}
Notice that these bounds are most effective when energy is small.

\subsection{Coding theory}

Our second setting is in coding theory.
Much is known about the interplay between the redundancy of a given code and
the amount of information that can be communicated per unit time; see Hall's
notes on coding \cite{HallCoding}, for a good introduction.
Here, we investigate messages transmitted through a noisy medium.

Consider a function $f:\mathcal{C}\rightarrow \mathcal{M}$, where $\mathcal{C}$
is the codespace and $\mathcal{M}$ is the message space.
In order to increase the likelihood that a message is decoded properly, even
with errors in transmission, we will often give a single message word more than
one code word.
That is, it will often be the case that $f(c) = f(c')$ for distinct
$c,c'\in\mathcal{C}$.
By definition, $V(f)$ will be precisely the number of distinct words in
$\mathcal{M}$, and $N_2(f)$ will be the number of times that two code words
represent the same message.

There are situations in which one has a particularly uneven message space,
where a small number of messages have high priority, and need the best chances
of being decoded correctly, while all remaining messages are less important,
and their incorrect decodings would have very little consequence.
For example, a message space between fire towers in a forest could have a small
number of special words about the existence or severity of a fire, and the
other words could describe other, less important details, like the weather, in
the case that there is no fire.
Similar applications exist in a variety of different contexts such as
operations in hostile environments.
In such situations, an application of Theorem \ref{seq2thm} yields the
following.
\begin{cor}
In a code with a codespace $\mathcal{C},$ a message space $\mathcal{M}$, an
assignment function $f:\mathcal{C} \rightarrow \mathcal{M}$, and
$t=|\lbrace f(c)=f(c'): c,c'\in \mathcal{C}, c\neq c' \rbrace|$,
we have
\begin{equation*}
n - \frac{t}{2} \le |\mathcal{C}| \le n - \frac{2t}{1+\sqrt{4t+1}}
\end{equation*}
\end{cor}
In this setting, our bounds can be viewed as providing a guide for balancing
between levels of redundancy and flexibility within the code.

\providecommand{\bysame}{\leavevmode\hbox to3em{\hrulefill}\thinspace}
\providecommand{\MR}{\relax\ifhmode\unskip\space\fi MR }
\providecommand{\MRhref}[2]{%
  \href{http://www.ams.org/mathscinet-getitem?mr=#1}{#2}
}
\providecommand{\href}[2]{#2}

\end{document}